\pdfoutput=1 
\documentclass[final]{siamart1116}
\usepackage[utf8]{inputenc}
\usepackage[T1]{fontenc}
\usepackage{bbm}
\usepackage{amssymb,amsmath}
\usepackage{amsfonts}
\usepackage{subfigure}
\usepackage{graphicx}
\usepackage{algpseudocode}
\usepackage{amsopn}
\usepackage{todonotes}

\graphicspath{{pics/}}
\newsiamthm{remark}{remark}
\newsiamthm{example}{example}

\numberwithin{theorem}{section}

\newcommand{\AbbrevTitle}{Computation of higher order cumulants} 
\newcommand{\TheTitle}{Efficient computation of higher order cumulant tensors} 
\newcommand{\TheAuthors}{Krzysztof Domino, Piotr Gawron, {\L}ukasz Pawela}

\headers{\AbbrevTitle}{\TheAuthors}

\title{{\TheTitle}\thanks{Submitted to the editors 07.03.2018.
		\funding{The research was partially financed by the National Science 
		Centre, Poland---project number 2014/15/B/ST6/05204. }}}

\author{
	Krzysztof Domino \thanks{Institute of Theoretical and Applied 
	Informatics, Polish Academy of Sciences, Ba{\l}tycka 5, 44-100 Gliwice, 
	Poland (\email{\{kdomino, gawron, lpawela\}@iitis.pl})}
	\and
	Piotr Gawron \footnotemark[2]
	\and
	{\L}ukasz Pawela \footnotemark[2]
}

\ifpdf
\hypersetup{
	pdftitle={\TheTitle},
	pdfauthor={\TheAuthors}
}
\fi

\newcommand{\ie}{\textit{i.e.}}

\newcommand{\R}{\mathbb{R}}
\renewcommand{\AA}{\mathcal{A}}
\newcommand{\BB}{\mathcal{B}}
\newcommand{\CC}{\mathcal{C}}
\newcommand{\MM}{\mathcal{M}}

\begin{document}
\maketitle
\date{March, 7, 2018}

\section*{Abstract}
In this paper, we introduce a novel algorithm for calculating arbitrary order
cumulants of multidimensional data. Since the $d$\textsuperscript{th} order
cumulant can be presented in the form of an $d$-dimensional tensor, the
algorithm is presented using tensor operations. The algorithm provided in the
paper takes advantage of super-symmetry of cumulant and moment tensors.
We show that the proposed algorithm considerably reduces the computational
complexity and the computational memory requirement of cumulant calculation as 
compared with existing algorithms. For the sizes of interest, the
reduction is of the order of $d!$ compared to the na\"ive algorithm.

\begin{keywords}
High order cumulants, non-normally distributed data, numerical algorithms
\end{keywords}

\begin{AMS}
	65Y05, 15A69, 65C60
\end{AMS}

\section{Introduction}
\subsection{Motivation} Cumulants of the order of $d > 2$ have recently started
to play an important role in the analysis of non-normally distributed
multivariate data. Some potential applications of higher-order cumulants include
signal filtering problems where the normality assumption is not required (see
\cite{geng2011research, latimer2003cumulant} and references therein). Another
application is finding the direction of received signals
\cite{porat1991direction, moulines1991second, cardoso1995asymptotic,
liang2009joint} and signal auto-correlation
analysis~\cite{manolakos2000systematic}. Higher-order cumulants are used in
hyper-spectral image analysis~\cite{geng2015joint}, financial data
analysis~\cite{arismendi2014monte,jondeau2015moment} and neuroimage
analysis~\cite{birot2011localization, becker2014eeg}. Outside the realm of
signal analysis, higher order cumulants can be applied to quantum noise
investigation purposes~\cite{gabelli2009high}, as well as to other types of
non-normally distributed data, such as weather data
\cite{cong2012interdependence, domino2014meteo, ozga2016snow}, various medical
data \cite{pougaza2008using}, cosmological data \cite{scherrer2009finance} or
data generated for machine learning purposes \cite{eban2013dynamic}.

In the examples mentioned above only cumulants of the order of $d \leq 4$ were
used due to growing computational complexity and large estimation errors of high
order statistics. The computational complexity and the use of computational
resources increases considerably with the cumulants' order by a factor of $n^d$,
where $d$ is the order of the cumulant and $n$ is the number of marginal
variables.

Despite the foregoing, cumulants of order $d=6$ of multivariate data were
successfully used in high-resolution direction-finding methods of multi-source
signals (the q-MUSIC algorithm) \cite{chevalier2006high, chevalier2005virtual,
chevalier2007higher, liu2008extended} despite higher variance of the statistic's
estimation. In such an algorithm, the number of signal sources that can be
detected is proportional to the cumulant's order \cite{pal2012multiple}.
Cumulants of the order of $d > 4$ also play an important role in financial data
analyses, as they enable measurement of the risk related to portfolios composed
of many assets \cite{rubinstein2006multi, martin2013consumption}. This is
particularly important during an economic crisis, since higher order cumulants
make it possible to sample larger fluctuation of prices
\cite{muzy2001multifractal}. In \cite{rubinstein2006multi}, cumulants of the
order of 2--6 of multi-asset portfolios were used as a measure of risk seeking
vs. risk aversion. In \cite{martin2013consumption}, it was shown that, during an
economic crisis, cumulants of the order of $d > 4$ are important to analyse
variations of assets and prices of portfolios. Further arguments for the utility
of cumulants of the order of $d > 4$ can be found in \cite{guizot2006hedge,
amin2003hedge} and \cite{domino2016use} where cumulant tensors of the order of
2--6 were used to analyse financial portfolios during an economic crisis.
Finally, let us consider the QCD (Quantum Chromodynamics) phase structure
research area. In \cite{friman2011fluctuations}, the authors have evidenced the
relevance of cumulants of the order of $5$ and $6$ of net baryon number
fluctuations for the analysis of freeze-out and critical conditions in heavy ion
collisions. Standard errors of those cumulant estimations were discussed in
\cite{luo2012error}.

In our study, we introduce an efficient method to calculate higher-order
cumulants. This method takes advantage of the recursive relation between
cumulants and moments as well as their super-symmetric structure. These features
enable us to reduce the computational complexity of the na\"ive algorithm and make
the problem tractable. In order to reduce complexity, we use the idea introduced
in \cite{schatz2014exploiting} to decrease the storage and computational
requirements by a factor of $O(d!)$.

This allows us to handle large data sets and overcome a major problem in
numerical handling of high order moments and cumulants. Consider that the
estimation error of the one-dimensional $d$\textsuperscript{th} central moment
is limited from above by $\sqrt{\frac{M_{2d}}{t}}$ where $M_{2d}$ is the
$(2d)$\textsuperscript{th} central moment and $t$ is number of data samples.
This is discussed further in~Appendix~\ref{app::estimation}. Consequently, the
accurate estimation of statistics of the order of $d > 4$ requires
correspondingly large data sets. In practice, our approach allows us to handle 
cumulants up to the tenth order.

\subsection{Normally and non-normally distributed data}
Let us consider the $n$-dimensional normally distributed random variable
$\mathbf{X}\sim\ \mathcal{N}(\mu, \Sigma)$ where $\Sigma$ is a~positive-definite
covariance matrix and $\mathbf{\mu}$ is a~mean value vector. In this case, the
characteristic function $\tilde{\phi}(\tau)$ and cumulant generating function
$K(\tau)$~\cite{kendall1946advanced,lukacs1970characteristics} are
\begin{equation}\label{eq::h1}
\begin{split}
\tilde{\phi}:\mathbb{R}^{n} \to \mathbb{R} \;& \ \ \tilde{\phi}(\tau)  = 
\exp\left(\tau^{\intercal}\mathbf{\mu}
+\frac{1}{2}\tau^{\intercal} \Sigma \tau\right), \\
K:\mathbb{R}^{n} \to \mathbb{R} \;& \ \ 
K(\tau)  = \log(\tilde{\phi}(\tau)) =
\tau^{\intercal}\mathbf{\mu} +\frac{1}{2}\tau^{\intercal} \Sigma 
\tau.
\end{split}
\end{equation}
It is easy to see that $K(\tau)$ is quadratic in $\tau$, and therefore  its
third and higher derivatives with respect to $\tau$ are zero. As we will in the 
next section, this implies that cumulants of order greater than two are equal 
to zero.
	
If data is characterised by a frequency distribution other than the multivariate
normal distribution, the characteristic function may be expanded in more terms
than quadratic, and cumulants of the order higher than two may have non-zero
elements. This is why they are helpful in distinguishing between normally and
non-normally distributed data or between data from different non-normal
distributions.

\subsection{Basic definitions}{\label{s::def}}
Let us start with a random process generating discrete $n$ dimensional
values. A sequence of $t$ samples of an $n$ dimensional random variable is
represented in the form of matrix $\mathbf{X}\in \R^{t \times n}$ such that
\begin{equation}\label{eq::variable}
	\mathbf{X} = \left[ \begin{array}{ccc}
		x_{1,1} & \dots & x_{1,n}  \\ 
		\vdots & \ddots & \vdots \\ 
		x_{t,1} & \dots & x_{t,n}  \\ 
	\end{array}   \right].
\end{equation}
This matrix can be represented as a sequence of vectors of realisations of $n$
marginal variables $X_i$
\begin{equation}\label{eq::mvariable}
	\mathbf{X} = \left[X_{1}, \dots, X_i, \dots, X_{n} \right],
\end{equation}
where 
\begin{equation}\label{eq::mvvariable}
	X_i = \left[x_{1, i}, \dots, x_{j, i}, \dots, x_{t, i} \right]^{\intercal}.
\end{equation}

In order to study moments and cumulants of $\mathbf{X}$, we need the notion of
super-symmetric tensors. Let us first denote the set $\{1, 2, \ldots, d\}$ as
$1:d$, a permutation of tuple $\mathbf{i}=(i_1, \ldots, i_d)$ as
$\pi(\mathbf{i})$.
\begin{definition}{\label{d::symmetry}}
Let $\AA \in \mathbb{R}^{\overbrace{n \times \cdots \times n}^d}$ be a tensor
with elements $a_\mathbf{i}$ indexed by multi-index
$\mathbf{i}=(i_1,\ldots,i_d)$. Tensor $\AA$ is super-symmetric iff it is
invariant under any permutation $\pi$ of the multi-index, \ie\
	\begin{equation}
		\forall_{\pi} \;\; 
		a_{\mathbf{i}} =
		a_{\pi(\mathbf{i})}.
	\end{equation}
\end{definition}

Henceforth we will write $\AA \in \R ^{[n,d]}$ for super-symmetric tensor 
$\AA$. A list
of all notations used in this paper is provided in Table~\ref{tab::symbols}.

\begin{definition}{\label{d::exp}}
	Let $\mathbf{X} \in \mathbb{R}^{t \times n}$ be as in 
	Eq.~\eqref{eq::variable}.
	We define the $d$\textsuperscript{th} moment as tensor $\MM_d(\mathbf{X}) 
	\in \R ^{[n, 
	d]}$. Its elements are indexed by multi-index $\textbf{i} = (i_1, \ldots, 
	i_d)$ and equal
	\begin{equation}{\label{d::moment}}
		m_{\mathbf{i}} = E(X_{i_1}, \ldots,
		X_{i_d}) = \frac{1}{t} \sum_{l = 1}^t \prod_{k=1}^{d} x_{l,i_k},
	\end{equation}
	where $E(X)$ is the expectational value operator and $X_{i_k}$ a vector of 
	realisations of the $i_k$\textsuperscript{th} marginal variable.
\end{definition}

\begin{definition}{\label{d::cent}}
	Let $\mathbf{X} \in \mathbb{R}^{t \times n}$ be as in 
	Eq.~\eqref{eq::variable}. 
	We define centered variable $\tilde{\mathbf{X}}\in \mathbb{R}^{t \times 
	n}$ as
	\begin{equation}
		\tilde{\mathbf{X}}= [\tilde{X_1}, \ldots, 
		\tilde{X}_i, 
		\ldots, \tilde{X}_n],\label{eq:centered}
		\text{ with }
		\tilde{X}_i = X_i - E(X_i).
	\end{equation}
\end{definition} 
\begin{table}[h]
	\renewcommand{\arraystretch}{1.2}
	\centering
	\begin{tabular}{lp{0.44\textwidth}}
		\textbf{Symbol} & \textbf{Description/explanation} \\ \hline 
		$\mathbf{i} =
		(i_1, \ldots, i_d)$ & $d$ element multi-index\\ \hline $|\mathbf{i}| = 
		d$ &
		size of multi-index (number of elements)\\ \hline $\pi(\mathbf{i})$ &
		permutation of multi-index \\ \hline $1:d$  & set of integers $\{1, 2, 
		\ldots,
		d\}$ \\ \hline $\mathbf{X} \in \mathbb{R}^{t \times n}$ & matrix of $t$
		realisations of $n$ dimensional random variable\\ \hline $X_i ={[x_{1, 
		i},
		\ldots, x_{t, i}]}^{\intercal}$ & vector of $t$ realisations of the
		$i$\textsuperscript{th} marginal random variable\\ \hline $E(X_{i_1}, 
		\ldots,
		X_{i_d}) = \frac{1}{t} \sum_{l = 1}^t \prod_{k=1}^{d} x_{l,i_k}$ &
		expectational value operator\\ \hline $\AA \in \R^{[n, d]}$ & 
		super-symmetric
		$d$ mode tensor of size $n \times \ldots \times n$, with elements 
		$a_{\mathbf{i}}$\\ 
		\hline $\AA
		\in \R^{n_1 \times \cdots \times n_d}$ & $d$ mode tensor of sizes
		$n_1 \times \ldots \times n_d$, with elements $a_{\mathbf{i}}$\\ \hline
		$\tilde{\mathbf{X}} \in \mathbb{R}^{t \times n}$ & matrix of $t$ 
		realisations
		of $n$ dimensional centered random variable\\ \hline $\CC_d(\mathbf{X}) 
		\in
		\mathbb{R}^{[n, d]}$ & the $d$\textsuperscript{th} cumulant tensor of
		$\mathbf{X}$ with elements $c_{\mathbf{i}}$\\ \hline 
		$(\CC_d)_{\mathbf{j}} \in
		\mathbb{R}^{[b, d]}$, $(\MM_d)_{\mathbf{j}} \in \mathbb{R}^{[b, d]}$ & 
		block
		of the $d$\textsuperscript{th} cumulant or moment tensor indexed by
		$\mathbf{j}$ in the block structure.\\ \hline 
		$\MM_d(\tilde{\mathbf{X}}) \in
		\mathbb{R}^{[n, d]}$ & the $d$\textsuperscript{th} central moment 
		tensor of
		$\mathbf{X}$ with elements $m_{\mathbf{i}}$\\ \hline $M_d(X) \in \R$ & 
		the
		$d$\textsuperscript{th} moment of one dimensional $X \in \R^t$
	\end{tabular}
	\caption{Symbols used in the paper.}\label{tab::symbols}
\end{table}

The first two cumulants respectively correspond to the mean vector and the
symmetric covariance matrix of $\mathbf{X}$. Given the following $K(\tau)$ 
estimator:
\begin{equation}\label{eq:cum_ge}
K(\tau) = 
\log\left(\frac{\sum_{j=1}^t\exp\left(\left[x_{j,1}, \ldots,  
	x_{j,n} \right] \cdot\tau \right)}{t}\right),
\end{equation} 
we first introduce the definition of cumulants of arbitrary order and later 
explicitly state definitions for cumulants of order one to 
four~\cite{kendall1946advanced,lukacs1970characteristics}

\begin{definition}{\label{d::cumulants}} Let $(i_1, \ldots, i_d)$ be a 
multi-index with elements $i_k \in 1: n$, and $K(\tau)$ the cumulant generation 
function of a given distribution. The $d$\textsuperscript{th} cumulant 
element is defined by \cite{kendall1946advanced,lukacs1970characteristics} 
\begin{equation}
c_{i_1, \ldots, i_d} = \frac{\partial^d}{\partial \tau_{i_1}, 
\ldots, \partial \tau_{i_d}} \log\left(K(\tau)\right) 
\bigg{|}_{\tau = 0},
\end{equation}
we drop an imaginary unit in definition for a presentation clarity. 
\end{definition}

\begin{definition}We define the first cumulant $\CC_1 \in \R^{[n,1]}$  as
\begin{equation}
	\CC_1(\mathbf{X}) = \left[E(X_1), \ldots,  E(X_n) \right].
\end{equation}
\end{definition}
\begin{definition}

We define the second cumulant $\CC_2 \in \R^{[n, 2]}$ as
\begin{equation}
	\CC_2(\mathbf{X}) 
	= \left[ \begin{array}{ccc}
		E\bigg(\tilde{X}_1\tilde{X}_1\bigg)  & \dots & 
		E\bigg(\tilde{X}_1\tilde{X}_n\bigg)  \\ 
		\vdots & \ddots & \vdots \\ 
				E\bigg(\tilde{X}_n\tilde{X}_1\bigg)  & \dots & 
		E\bigg(\tilde{X}_n\tilde{X}_n\bigg) \\ 
	\end{array}   \right].
\end{equation}
\end{definition}
\begin{definition} We define the third cumulant as a three-mode tensor $\CC_3 
\in \mathbb{R}^{[n, 3]}$ with elements
\begin{equation}\label{eq::c3}
	c_{\mathbf{i}}(\mathbf{X}) = 
	E\left(\tilde{X}_{i_1}\tilde{X}_{i_2}\tilde{X}_{i_3}\right). 
\end{equation}
\end{definition}

Cumulants of order greater than three can be computed from moments
\cite{mccullagh2009cumulants, mccullagh1987tensor}, however the relation is
complex and requires a special notation which is introduced in
Subsection~\ref{ssec::partitions}. To show how complicated the formulas might 
become we state here the partial formula for the fourth cumulant.

\begin{definition} We define the fourth cumulant as a four-mode tensor $\CC_4 
\in
\mathbb{R}^{[n, 4]}$ with elements

\begin{equation}\label{eq::c4n}
\begin{split}
c_{\mathbf{i}}(\mathbf{X}) &= 
E\left(X_{i_1}X_{i_2}X_{i_3}X_{i_4}\right)	
\underbrace{-E\left(X_{i_1}\right)E\left(X_{i_2}X_{i_3}X_{i_4}\right) 
-E\left(X_{i_2}\right)E\left(X_{i_1}X_{i_3}X_{i_4}\right) - \ldots}_{\times 4}
\\
&-\underbrace{E\left(X_{i_1}X_{i_2}\right)E\left(X_{i_3} X_{i_4}\right) - 
	\ldots}_{\times 
	3}+2\left(\underbrace{E\left(X_{i_1}\right)E\left(X_{i_2}\right)E\left(X_{i_3}X_{i_4}\right)+
\ldots}_{\times 6}\right)
 \\ &-6 E\left(X_{i_1}\right)E\left(X_{i_2}\right) 
E\left(X_{i_3}\right) E\left(X_{i_4}\right).
\end{split}
\end{equation}
Switching to the centered variable $\tilde{\mathbf{X}}$, using a fact that
$E(\tilde{X}_i) = 0$, and $c_{\mathbf{i}}(\mathbf{X}) =
c_{\mathbf{i}}(\tilde{\mathbf{X}})$ for $|\mathbf{i}| \geq 2$, cumulants of
order greater than one are mean shift invariant~\cite{mccullagh1987tensor}, we
can write Eq.~\eqref{eq::c4n} in a following manner:
\begin{equation}\label{eq::c4}
\begin{split}
c_{\mathbf{i}}(\mathbf{X}) &= 
E\left(\tilde{X}_{i_1}\tilde{X}_{i_2}\tilde{X}_{i_3}\tilde{X}_{i_4}\right)	
-E\left(\tilde{X}_{i_1}\tilde{X}_{i_2}\right)E\left(\tilde{X}_{i_3}\tilde{X}_{i_4}\right)
\\
&-E\left(\tilde{X}_{i_1}\tilde{X}_{i_3}\right)E\left(\tilde{X}_{i_2}\tilde{X}_{i_4}\right)
-E\left(\tilde{X}_{i_1}\tilde{X}_{i_4}\right)E\left(\tilde{X}_{i_2}\tilde{X}_{i_3}\right).
\end{split}
\end{equation}

\end{definition}

\begin{remark}
Each cumulant tensor $\CC_d$ as well as each moment tensor $\MM_d$ is
super-symmetric \cite{barndorff1989asymptotic}.
\end{remark}
As the formula for a cumulant of an arbitrary order is very complex, our core
result is the numerical handling of a cumulant and is discussed in depth in
Section~\ref{sec::cumulant}. To compute cumulant tensor of order $d$ we use
central moment tensors of order $2,3,d-2$ and $d$ and take advantage of cumulant
and moment tensors super-symmetry. Importantly we do not need to determine
$(d-1)$\textsuperscript{th} moment tensor.

\section{Moment tensor calculation}{\label{sec::moment}}
To provide a simpler example, we start with algorithms for calculation of the
moment tensor. Next, in Section~\ref{sec::cumulant}, those algorithms will be utilised to
recursively calculate the cumulants.

\subsection{Storage of super-symmetric tensors in block
structures}\label{ssec::storage}
In this section, we are going to follow the idea introduced by Schatz et
al.~\cite{schatz2014exploiting} concerning the use of blocks to store symmetric
matrices and super-symmetric tensors in an efficient way. To make the
demonstration more accessible, we will first focus on the matrix case. Let us 
suppose we have symmetric matrix $\CC_2 \in \mathbb{R}^{[n, 2]}$. We can store
the matrix in blocks and store only upper triangular blocks,
\begin{equation}
\CC_2 = \left[ \begin{array}{cccc}
({\CC_2})_{11} & ({\CC_2})_{12} & \cdots & ({\CC_2})_{1\bar{n}} \\ 
\text{NULL} & ({\CC_2})_{22} & \cdots &  ({\CC_2})_{2\bar{n}} \\ 
\vdots & \vdots & \ddots & \vdots \\
\text{NULL} & \text{NULL} & \cdots & ({\CC_2})_{\bar{n} \bar{n}} \\ 
\end{array}   \right],
\end{equation}
where NULL represents an empty block, and $\bar{n}=\lceil \frac{n}{b} \rceil$.
Entries below the diagonal do not need to be stored and calculated as they are
redundant. Each block $({\CC_2})_{j_1, j_2}: j_1 \leq j_2 \wedge j_2 < \bar{n}$
is of size $b \times b$. Blocks $({\CC_2})_{j_1,\bar{n}}: j_1 < \bar{n}$ are of
size $b \times b_l$, and block $({\CC_2})_{\bar{n}, \bar{n}}$ is of size $b_l
\times b_l$, where:
\begin{equation}
b_l = (n-b(\bar{n}-1)).
\end{equation}
This representation significantly reduces the overall storage footprint while
still providing opportunities to achieve high computational performance.

This representation can easily be extended for purposes of super-symmetric
tensors. Let us assume that $\CC_d \in \mathbb{R}^{[n, d]}$ is a super-symmetric
tensor. All data can be stored in blocks $({\CC_d})_{j_1, \ldots, j_d} \in
\mathbb{R}^{b_{j_1} \times \cdots \times b_{j_d}}$. If indices $j_1, \ldots,
j_d$ are not sorted in an increasing order, such blocks are redundant and
consequently replaced by NULL. Similarly to the matrix case we have
\begin{equation}
b_{j_p} =  \left\{ \begin{array}{ll}
b, & \textrm{if $j_p < \bar{n}$},\\
b_l, & \textrm{if $j_p = \bar{n}.$}
\end{array} \right.
\end{equation} 
In the subsequent sections we present algorithms for moment and cumulant tensor
calculation and storage. For simplicity, we assume that $b|n$ and $\bar{n} =
\frac{n}{b}$. The generalization is straightforward and, at this point, would
only obscure the main idea.

Henceforth each block is a hypercube of size $b^d$ and there are
$\binom{\bar{n}+d-1}{\bar{n}}$ such unique blocks \cite{schatz2014exploiting}.
Such storage scheme, proposed in \cite{schatz2014exploiting}, requires the
storage of $b^d \binom{\bar{n}+d-1}{\bar{n}}$ elements.

\subsection{The algorithm}\label{sec::algorithm}
In this and following sections, we present the moment and cumulant calculation
algorithms that use the block structure. To compute the $d$\textsuperscript{th}
moment tensor we use Def. \ref{d::exp}. Algorithm \ref{alg:center} computes a
single block of the tensor, while Algorithm \ref{alg:cm} computes the whole
tensor in the block structure form.

\begin{algorithm}
	\caption{A single block of central moment, used to perform Algorithm 
		\ref{alg:cm} \label{alg:center}}
	\begin{algorithmic}[1]
		\State \textbf{Input}: $\mathbf{X} \in \mathbb{R}^{t \times n}$	
		--- data matrix; $(j_1, \ldots, j_d)$ --- 
		multi-index of a 
		block; $b$ --- block's size.
		\State \textbf{Output}: $(\MM_d) \in \mathbb{R}^{[b, d]}$ --- 
		a single block of the block structure. 
		\Function{momentblock}{$\mathbf{X}, (j_1, \ldots, j_d), b$}
		\For{$i_1 \gets 1 \textrm{ to } b, \ldots, i_d \gets 1 \textrm{ to } 
			b$}
		\State $(m_{i_1, \ldots, i_d}) = \frac{1}{t} \sum_{l = 1}^{t} 
		\left(\prod_{k=1}^{d} \mathbf{X}_{l, (j_k-1)b+i_k}
		\right)$ 
		\EndFor
		\State \Return $(\MM_d)$
		\EndFunction
	\end{algorithmic}
\end{algorithm}	

\begin{algorithm}
	\caption{The $d$\textsuperscript{th} moment tensor stored as a block 
	structure
	\label{alg:cm}}
	\begin{algorithmic}[1]	
		\State 	\textbf{Input}: $\tilde{\mathbf{X}}  \in \mathbb{R}^{t	\times 
		d}$ --- data matrix; $d$ --- moment's 
		order; $b$ --- block's size.
\State \textbf{Output}: ${\MM_d} \in \R^{[n, d]}$ --- $d$\textsuperscript{th} 
moment tensor
stored as the block structure. \Function{moment}{$\mathbf{X}, d, b$}
\State $\bar{n} = \frac{n}{b}$ \For{$j_1 \gets 1 \textrm{ to } \bar{n}, j_{2}
	\gets j_1 \textrm{ to } \bar{n}, \ldots, j_{d} \gets j_{d-1} \textrm{ to }
	\bar{n}$} \State $(\MM_d)_{(j_1, \ldots, j_d)} = \textrm{\footnotesize
	MOMENTBLOCK \normalsize}(\mathbf{X}, (j_1, \ldots, j_d), b)$ \EndFor
\State \Return ${\MM_d}$ \EndFunction
	\end{algorithmic}
\end{algorithm}	

Based on \cite{schatz2014exploiting} and the discussion in the previous
subsection, we can conclude that reduction of redundant blocks reduces the
storage and computational requirements of the $d$\textsuperscript{th} moment
tensor by a factor of $d!$ for $d \ll n$ compared to the na\"ive algorithm. The
detailed analysis of the computational requirements will be presented in
Section~\ref{sec::performance}.
\subsection{Parallel computation of moment tensor}{\label{ssec::parallel}}
For large $t$, it is desirable to speedup the moment tensor calculation further.
This can be achieved via  a simple parallel scheme. Let us suppose for the sake
of simplicity, that we have $p$ processes available, and $p|t$. Starting with
data $\textbf{X} \in \R^{t \times n}$ we can split them into $p$ non overlapping
subsets $\textbf{X}_s \in \R^{\frac{t}{p} \times n}$. In the first step, for
each subset, we compute in parallel moment tensor $\MM_d(\mathbf{X}_s)$ using
Algorithm~\ref{alg:cm}. In the second step, we perform the following reduction
\begin{equation}\label{eq::mapred}
\MM_d(\textbf{X}) = 
\frac{1}{p}\sum_{s=1}^p\MM_d(\textbf{X}_s).
\end{equation}
The elements of the tensor under the sum on the RHS are
\begin{equation}{\label{d::momentp}}
m_{\mathbf{i}}(\mathbf{X}_s) = \frac{p}{t} \sum_{l = 
(\frac{t}{p}-1)s+1}^{\frac{ts}{p}} 
\prod_{k=1}^{|\mathbf{i}|} x_{l,i_k}.
\end{equation}
The element of the moment tensor of $\mathbf{X}$ is
\begin{equation}{\label{d::momentp1}}
m_{\mathbf{i}}(\mathbf{X}) = \frac{1}{t} \sum_{s = 1}^p \sum_{l = 
(\frac{t}{p}-1)s+1}^{\frac{ts}{p}} \prod_{k=1}^{|\mathbf{i}|} x_{l,i_k} = 
\frac{1}{p} \sum_{s=1}^p m_{\mathbf{i}}(\mathbf{X}_s).
\end{equation}
These steps are summarised in Algorithm~\ref{alg:pal}.
\begin{algorithm}
	\caption{Parallel computation of the $d$\textsuperscript{th} moment tensor
		\label{alg:pal}}
	\begin{algorithmic}[1]	
		\State 	\textbf{Input}: $\tilde{\mathbf{X}}  \in \mathbb{R}^{t	\times 
			n}$ --- data matrix; $d$ --- moment's order; $b$ --- block's size; 
			$p$	---	number of processes.
		\State \textbf{Output}: ${\MM_d} \in \R^{[n, d]}$ --- 
		$d$\textsuperscript{th} moment 
		tensor
		stored as the block structure. \Function{momentnc}{$\tilde{\mathbf{X}}, 
		d, 
		p, b$}
		\State $\mathbf{X} \rightarrow [\mathbf{X}_1, \ldots,  \mathbf{X}_s, 
		\ldots \mathbf{X}_p]: {(x_s)}_{l,i} = 
		x_{\left(\frac{t}{p}-1\right)s+l,i}$
		\For{$s \gets 1 \textrm{ to } p$} \Comment{perform in parallel}
		\State $\MM_d(\textbf{X}_s) = \textrm{\footnotesize
			MOMENT\normalsize}(\mathbf{X}_s, d, b)$ 
		\EndFor
		\State $\MM_d(\textbf{X}) = 
		\frac{1}{p}\sum_{s=1}^p\MM_d(\textbf{X}_s)$ \Comment{reduce scheme}
		\State \Return ${\MM_d}$ \EndFunction
	\end{algorithmic}
\end{algorithm}	
\section{Calculation of cumulant tensors}\label{sec::cumulant}
At this point, we can define our main result, \ie~an algorithm for calculating 
cumulants of arbitrary order of multi-dimensional data.
\subsection{Index partitions and permutations}\label{ssec::partitions}
In this section, we present a recursive formula that can be used to calculate 
the $d$\textsuperscript{th} cumulant of $\mathbf{X}$. We begin with some
definitions, mainly concerning combinatorics, before discussing the general
formula.

\begin{definition}{\label{d::part}} Let $\mathbf{k} = (k_1, \ldots, k_d): k_i = 
i$, and $ \sigma \in 1:d$. Partition $P_{\sigma}(\mathbf{k})$ of tuple
$\mathbf{k}$ is the division of $\mathbf{k}$ into $\sigma$ non-crossing
sub-tuples: $P_{\sigma}(\mathbf{k}) = (\mathbf{k}_1, \ldots,
\mathbf{k}_{\sigma})$,
	\begin{equation}
	\bigcup_{r=1}^{\sigma} \mathbf{k}_r = \mathbf{k} \ \wedge\ \forall_{r \neq 
	r'} \
	\mathbf{k}_r \cap 
	\mathbf{k}_r' = \emptyset. 
	\end{equation}
\end{definition}
In what follows, we will denote the permutations of a tuple of tuples 
$(\mathbf{i}_1,
\ldots, \mathbf{i}_{\sigma})$, as $\pi'(\mathbf{i}_1, \ldots,
\mathbf{i}_{\sigma})$.
\begin{definition}\label{d::class}{$[P_{\sigma}(\mathbf{k})]$ --- the 
representative of the equivalence class of partitions.}
	Let $ P_{\sigma}(\mathbf{k}) = (\mathbf{k}_1, \ldots, \mathbf{k}_{\sigma})$ 
	and $P'_{\sigma}(\mathbf{k}) = (\mathbf{k}'_1, \ldots, 
	\mathbf{k}'_{\sigma})$ be partitions of $\mathbf{k}$. Let us introduce the 
following equivalence relation:
\begin{equation}
	P_{\sigma}(\mathbf{k}) \sim P'_{\sigma}(\mathbf{k}) \Leftrightarrow
	\Big(\exists_{\pi'} \ \forall_{r \in 1:\sigma} \ \exists_{\pi_r}:
	(\mathbf{k}_1,\ldots, \mathbf{k}_\sigma) = 
	\pi'\left(\pi_{1}(\mathbf{k}'_{1}),
	\ldots, \pi_{\sigma}(\mathbf{k}'_{\sigma})\right) \Big).
\end{equation}
	This relation defines the equivalence class. Henceforth we will take only 
	one representative of each equivalence class and denote it as 
	$[P_{\sigma}(\mathbf{k})]$. The representative will be such that all
	$\mathbf{k}_r$ are sorted in an increasing order. We will denote a set of 
	all
	such equivalence classes as $\{[P_{\sigma}(\mathbf{k})]\}$.
\end{definition}
\begin{remark}
	The number of partitions of set $\mathbf{k}$ of size $d$ into $\sigma$ 
	parts is given by the Stirling Number of the second kind, 
	\cite{grahamconcrete}
	\begin{equation}{\label{eq::n1}}
		\# \{[P_{\sigma}(\mathbf{k})]\}  = S(d, \sigma) = \frac{1}{\sigma !} 
		\sum_{j=0}^{\sigma} (-1)^{(\sigma-j)} \binom{\sigma}j j^d.
	\end{equation}
\end{remark}

\begin{definition}{\label{d::mop}}
	Consider tensors $\CC_{d_1} \in \mathbb{R}^{[n, d_1]}$, $\CC_{d_2} \in
	\mathbb{R}^{[n, d_2]}$ indexed by $\mathbf{i}$ and $\mathbf{i'}$ 
	respectively.
	Their outer product $\CC_{d_1}\otimes\CC_{d_2}= \AA_{d_1+d_2} \in
	\mathbb{R}^{\overbrace{n\times\ldots\times n}^{d_1+d_2}}$ is defined as
	\begin{equation}
	 	a_{(\mathbf{i}, \mathbf{i'})} = c_{\mathbf{i}} c_{\mathbf{i'}},
	\end{equation}
	where $(\mathbf{i}, \mathbf{i'})$ denotes multi-index $(i_1, \ldots, 
	i_{d_1}, i'_1, \ldots, i'_{d_2})$.	
\end{definition}
As an example consider the outer product of symmetric matrix $\CC_2$ by itself: 
$\AA_4 = \CC_2 \otimes \CC_2$, that is
only partially symmetric, $a_{i_1, i_2, i_3, i_4} = c_{i_2, i_1, i_3, i_4} =
c_{i_1, i_2 i_4, i_3} = c_{i_3, i_4, i_1, i_2}$, but in general $c_{i_1,
	i_2, i_3, i_4} \neq c_{i_1, i_3, i_2, i_4} \neq c_{i_1, i_4, i_3, i_2}$.
To obtain a super-symmetric outcome of the outer product of super-symmetric 
tensors, we need to apply the following symmetrisation procedure.

\begin{definition}{\label{d::k}}{The sum of outer products of super-symmetric 
tensors.}
Let $\AA_d \in \mathbb{R}^{[n,d]}$ be a tensor indexed by $\mathbf{i} = 
(i_1, 
\ldots, i_d)$. Let $\mathbf{k}_r$ be a sub-tuple of its modes 
according to Def. \ref{d::class}, and let $\mathbf{i}_{\mathbf{k}_r} =
\left(i_{(\mathbf{k}_r)_{1}}, \ldots, 
i_{(\mathbf{k}_r)_{|\mathbf{k}_r|}}\right)$. For the given $\sigma$, we 
define the sum of outer products of $\CC_{d_r} \in \mathbb{R}^{[n, d_r]}$ 
where $r \in 1: \sigma$ and $\sum_{r=1}^{\sigma} d_r = d$, using the 
elementwise notation, as 
\begin{equation}
a_{\mathbf{i}} = \sum_{\zeta\in \{[P_{\sigma}(1 : d)]\}} 
\prod_{\mathbf{k}_r \in \zeta} 
c_{\mathbf{i}_{\mathbf{k}_r}}.
\end{equation}
We will use the following abbreviation using tensor notation
\begin{equation}{\label{eq::am}}
\AA_d = \sum_{\zeta\in \{[P_{\sigma}(1 : d)]\}} \bigotimes_{\mathbf{k}_r 
\in \zeta} \CC_{(\mathbf{k}_r)}.
\end{equation}
\end{definition}

Consider $\AA_d$ as in Eq.~\eqref{eq::am} where $\CC_{d_r} \in \R^{[n, d_r]}$
are super-symmetric and $\CC_{d_r} = \CC_{d_{r'}}$ iff $d_r = d_{r'}$ and
$\mathbf{i}$ is a multi-index of $\AA_d$. The sum over all representatives of 
equivalence classes $\{[P_{\sigma}(1:d)]\}$ fully symmetrises the outer
product, and therefore $\AA_d$ is super-symmetric. In other words, due to the
super-symmetry, any permutation of multi--index $\mathbf{i}$ of $\AA_d$
that leads only to a permutation of indices inside some $\CC_{d_r}$ refers to
the same value of $\AA_d$. Any permutation of $\mathbf{i}$ that leads only to
the switch between $\CC_{d_r}$ and $\CC_{d_{r'}}$ inside an outer product in
Eq.~\eqref{eq::am} also refers to the same value of $\AA_d$. Any other
permutation of $\mathbf{i}$ that cannot be represented as above switches
between equivalence classes as well, and so it switches between elements of sum 
Eq.~\eqref{eq::am} and  refers to the same value of $\AA_d$.

\begin{example}{\label{d::mops}}
	Consider $\CC_1 \in \R^{[n,1]}, \CC_2 \in \mathbb{R}^{[n, 2]}, 
	\CC_3 \in 
	\mathbb{R}^{[n, 3]}$, and 
	$\AA_4 \in \mathbb{R}^{[n, 4]}$ such that
	\begin{equation}
		\AA_4 = \sum_{\zeta \in \{[P_{2}(1:4)]\}}
		\bigotimes_{\mathbf{k}_r\in\zeta}\CC_{(\mathbf{k}_r)},
	\end{equation}
	then 
	\begin{equation}
	\begin{split}
		a_{i_1, i_2, i_3, i_4}
		&= c_{i_1, i_2}c_{i_3, i_4}+c_{i_1, 
		i_3}c_{i_2, i_4}+c_{i_1, i_4}c_{i_2, i_3} \\ &\phantom {=\ 
		}+c_{i_1}c_{i_2,i_3,i_4} + 
		c_{i_2}c_{i_1,i_3,i_4} + c_{i_3}c_{i_1,i_2,i_4} + 
		c_{i_4}c_{i_1,i_2,i_3},
		\end{split}
	\end{equation}
such $\AA_4$ is super-symmetric, since there is no permutation of $(i_1, i_2, 
i_3, i_4)$ that changes its elements, \ie \ $a_{i_1, i_2, i_3, i_4} 
= a_{i_2, i_1, i_3, i_4} = a_{i_3, i_2, i_1, i_4} = a_{i_3, i_4, i_1, i_2} = 
\ldots$ .
\end{example}
\subsection{Cumulant calculation formula}{\label{s::f}}
The following recursive relation can be used to relate moments and cumulants of 
$\mathbf{X}$:
\begin{equation}\label{eq::cumt}
\mathbb{R}^{[n,d]} \ni \MM_d(\mathbf{X}) =  \sum_{\sigma = 1}^{d} 
\sum_{\zeta \in \{[P_{\sigma}(1:d)]\}} 
\bigotimes_{\mathbf{k}_r 
\in \zeta} \CC_{(\mathbf{k}_r)}(\mathbf{X}).
\end{equation}
This can be written in an elementwise manner as 
in~\cite{barndorff1989asymptotic}
\begin{equation}\label{eq::cum}
m_{\mathbf{i}}(\mathbf{X}) = \sum_{\sigma = 1}^{d} \sum_{\zeta \in 
\{[P_{\sigma}(1:d)]\}} \prod_{\mathbf{k}_r \in \zeta} 
c_{\mathbf{i}_{\mathbf{k}_r}}(\mathbf{X}).
\end{equation}
For the sake of completeness, we present an alternative proof of
Eq.~(\ref{eq::cum}) in Appendix~\ref{app::recurrence}.

In order to compute $\CC_d(\mathbf{X})$, let us consider the case where $\sigma 
= 1$ separately. By definition, $[P_{\sigma = 1}(1:d)] = (1,\ldots, d)$, so:
\begin{equation}\label{eq::ctn}
\MM_d(\mathbf{X}) = \CC_d(\mathbf{X}) + \sum_{\sigma = 2}^{d} 
\sum_{\zeta \in \{[P_{\sigma}(1:d)]\}}  
\bigotimes_{\mathbf{k}_r 
\in \zeta} \CC_{(\mathbf{k}_r)}(\mathbf{X}).
\end{equation}
The $d$\textsuperscript{th} cumulant tensor can be calculated given the
$d$\textsuperscript{th} moment tensor and cumulant tensors of the order of $r 
\in 1:(d-1)$
\begin{equation}\label{eq::cum_rec_gen}
\mathbb{R}^{[n,d]} \ni \CC_d(\mathbf{X}) = \MM_d(\mathbf{X}) - \sum_{\sigma =
2}^{d} \sum_{\zeta \in \{[P_{\sigma}(1:d)]\}} \bigotimes_{\mathbf{k}_r \in
\zeta} \CC_{(\mathbf{k}_r)}(\mathbf{X}) .
\end{equation}

To simplify Eq.~\eqref{eq::cum_rec_gen}, let us observe that cumulants of the 
order of two or higher for a non-centered variable and a centered variable are 
equal. The
first order cumulant for a centered variable is zero. Hereafter, we introduce
partitions into sub-tuples of size larger than one.
\begin{definition}{\label{d::part2}}Let $\mathbf{k} = (1, \ldots, d)$, 
	and $ \sigma \in 1:d$. The at least two element partition 
	$P_{\sigma}^{(2)}(\mathbf{k})$ of 
	tuple $\mathbf{k}$ is the division of $\mathbf{k}$ into $\sigma$ 
	sub-tuples: $P_{\sigma}^{(2)}(\mathbf{k}) = (\mathbf{k}_1, \ldots, 
	\mathbf{k}_{\sigma})$, such that
	\begin{equation}
		\bigcup_{r=1}^{\sigma} \mathbf{k}_r = \mathbf{k} \ \wedge \ \forall_{r 
		\neq r'} \  
		\mathbf{k}_r \cap \mathbf{k}_r' = \emptyset \ \wedge \ \forall_{r} \ 
		|\mathbf{k}_r| \geq 2.
	\end{equation}	
\end{definition}
The definition of the representative of equivalence class
$[P^{(2)}_{\sigma}(\mathbf{k})]$  and the set of such representatives
$\{[P^{(2)}_{\sigma}(\mathbf{k})]\}$ are analogous to Def. \ref{d::class}.
Consequently, we can derive the final formula

\begin{equation}\label{eq::cum_r}
\begin{split}
\CC_d(\mathbf{X}) = \CC_d(\mathbf{\tilde{X}}) &= \MM_d(\mathbf{\tilde{X}}) 
- \sum_{\sigma = 2}^{d}	\sum_{\zeta \in \{[P_{\sigma}(1:d)]\}} 
\bigotimes_{\mathbf{k}_r \in \zeta} 
\CC_{(\mathbf{k}_r)}(\mathbf{\tilde{X}}) \\&= 
\MM_d(\mathbf{\tilde{X}}) - \sum_{\sigma = 2}^{\sigma_{\max}} 
\sum_{\zeta \in \{[P^{(2)}_{\sigma}(1:d)]\}} 
\bigotimes_{\mathbf{k}_r 
\in \zeta} \CC_{(\mathbf{k}_r)}(\mathbf{X}).
\end{split}
\end{equation}
Let us determine the $\sigma_{\max}$ limit. If $d$ is even, it can be divided
into at most $\sigma_{\max} = \frac{d}{2}$ parts of size two; if $d$ is odd, it
can be divided into at most $\sigma_{\max} = \frac{d-1}{2}$ parts:
$\frac{d-1}{2}-1$ parts of size two and one part of size three. Hence we can
conclude that $\sigma_{\max}=\lfloor\frac{d}{2}\rfloor$.

As a simple example, consider the cumulants of the order of three and four. 
Since
$\forall_{\sigma} \{[P^{(2)}_{\sigma}(1:3)]\} = \emptyset \ \wedge \
\{[P^{(2)}_{\sigma}(1:2)]\} = \emptyset$, then $\CC_2(\mathbf{X}) =
\MM_2(\mathbf{\tilde{X}})$ and $\CC_3(\mathbf{X}) = \MM_3(\mathbf{\tilde{X}})$,
\ie~the second cumulant matrix and the third cumulant tensor are simply the
second and the third central moments. Formulas for cumulant tensors of the 
order greater than three are more complicated. For example, consider the
$4$\textsuperscript{th} cumulant tensor
\begin{equation}\label{eq::ct4}
\mathbb{R}^{[n, 4]} \ni
\CC_4(\mathbf{X}) = \MM_4(\mathbf{\tilde{X}}) - \sum_{\zeta \in 
\{[P_{2}^{(2)}(1:4)]\}} 
\bigotimes_{\mathbf{k}_r\in\zeta}\CC_{\mathbf{k}_r}(\mathbf{X}).
\end{equation}
Using the elementwise notation, where $\mathbf{i} = (i_1, i_2, i_3, i_4)$, we 
have
\begin{equation}{\label{eq:c4}}
	c_{\mathbf{i}}(\mathbf{X}) = 
	m_{\mathbf{i}}(\tilde{\mathbf{X}}) -  
	c_{i_1, i_2}(\mathbf{X}) c_{i_3, i_4}(\mathbf{X}) -  
	c_{i_1, i_3}(\mathbf{X}) c_{i_2, i_4}(\mathbf{X}) -  
	c_{i_1, i_4}(\mathbf{X}) c_{i_2, i_3}(\mathbf{X}).
\end{equation}

\subsection{Algorithms to compute cumulant tensors}\label{ssec::alg}
Let us suppose that $(\BB_{\mathbf{i}})_{\mathbf{j}}$ is the
$\mathbf{i}$\textsuperscript{th} element of the $\mathbf{j}$\textsuperscript{th}
block of the super--symmetric tensor of the order of $|\mathbf{i}| = 
|\mathbf{j}| = d$.
Similarly, $(\CC_{\mathbf{i}_{\mathbf{k}}})_{\mathbf{j}_{\mathbf{k}}}$ is the
$\mathbf{i}_{\mathbf{k}}$\textsuperscript{th} element of the
$\mathbf{j}_{\mathbf{k}}$\textsuperscript{th} block of the
$|\mathbf{i}_{\mathbf{k}}|=|\mathbf{j}_{\mathbf{k}}|$\textsuperscript{th}
cumulant tensor according to Def.~\ref{d::k}---we skip now $r$ in $\mathbf{k}_r$
for brevity. With reference to Def.~\ref{d::k} and Def.~\ref{d::class}, 
$\mathbf{k}$ is always sorted and from the properties of the block structure 
$\mathbf{j}$ is also sorted, hence $\mathbf{j}_{\mathbf{k}}$ is sorted as well. 
To determine
$\{[P^{(2)}_{\sigma}(1:d)]\}$ we use modified Knuth's algorithm $7.2.1.5H$
\cite{knuth1998art}. Now we have all the components to introduce
Algorithm~\ref{alg:mlc} which computes a super-symmetric sum of outer products
of lower order cumulants. Algorithm~\ref{alg:mlc} computes the inner sum of
Eq.~\eqref{eq::cum_r} and takes advantages of the super-symmetry of tensors by 
using the block structure.

Finally, Algorithm~\ref{alg:ncc} computes the $d$\textsuperscript{th} cumulant
tensor. It uses Eq.~\eqref{eq::cum_r} to calculate the cumulants and
importantly takes advantage of the super-symmetry of tensors, because it refers 
to Algorithm~\ref{alg:cm} (moment tensor calculation) and 
Algorithm~\ref{alg:mlc} that both use the block structure.

\begin{algorithm}
	\caption{Sum of outer products of the $\sigma$ cumulants 
		\label{alg:mlc}}
	\begin{algorithmic}[1]	
		\State \textbf{Input} $d$ --- order of output; $\sigma$ --- number of 
		subsets for partitions; $\CC_{2}, \ldots, 
		\CC_{r} \in \R^{[n, r]}, 
		\ldots, \CC_{d - 2}$ --- lower cumulants stored as a block 
		structure. 	
		\State \textbf{Output} $\BB \in \R^{[n, d]}$ -- sum of outer products
		stored as the block structure.
		\Function{outerpodcum}{${d, \sigma, \CC_{2}, \ldots, \CC_{d - 2}}$}
		\Comment{$\bar{n} = \frac{n}{b}$, $b$ --- block size}
		\For{$i_1 \gets 1 \textrm{ to } b, \ldots, i_d \gets 1 \textrm{ to } b$}
		\For{$j_1 \gets 1 \textrm{ to } \bar{n}, j_{2} \gets j_1 \textrm{ to } 
		\bar{n}, \ldots, j_{d} \gets j_{d-1} \textrm{ to } \bar{n}$} 
		\State 
		\begin{equation*} 
		\begin{split}
		\mathbf{i} &= (i_1, \ldots, i_d) \ \ \  \mathbf{j} = (j_1, \ldots, j_d) 
		\\ (\BB_{\mathbf{i}})_{\mathbf{j}} &= \sum_{\zeta \in 
		\{[P^{(2)}_{\sigma}(1:d)]\}} \prod_{\mathbf{k} \in \zeta} 
		\left(C_{\mathbf{i}_{\mathbf{k}}}\right)_{\mathbf{j}_{\mathbf{{k}}}}
		\end{split}
		\end{equation*}
	\EndFor
		\EndFor
		\State \Return $\BB$
		\EndFunction	
	\end{algorithmic}
\end{algorithm}	

\begin{algorithm}
	\caption{The $d$\textsuperscript{th} cumulant, using  
		Eq.~\eqref{eq::cum_r}.
		\label{alg:ncc}}
	\begin{algorithmic}[1]
		\State	\textbf{Input}: $\tilde{\mathbf{X}} \in \mathbb{R}^{t \times 
		n}$ --- 
		matrix of centered data; 
		$d \geq 3$ --- order; $\CC_{2}, \ldots, \CC_{d - 
			2}$ --- lower cumulants structured in blocks of size 
		$b^i$ each.
		\State \textbf{Output}: $\CC_d \in \R^{[n, d]}$ --- the 
		$d$\textsuperscript{th} 
		cumulant stored as the block structure.
		\Function{cumulant}{$\tilde{\mathbf{X}}, d, \CC_{2}, \ldots, \CC_{d 
		- 2}$}
		\State
		$\CC_d = \textrm{\footnotesize MOMENT\normalsize}(\tilde{\mathbf{X}}, 
		d, b) - \sum_{\sigma = 
		2}^{\lfloor\frac{d}{2}\rfloor}\textrm{\footnotesize 
		OUTERPRODCUM\normalsize}(d, 
		\sigma, \CC_{2}, 
		\ldots, \CC_{d - 2})$
		\State \Return $\CC_d$ 
		\EndFunction
	\end{algorithmic}
\end{algorithm}	

\section{Implementation}\label{sec::implementation} All algorithms presented in this paper 
are implemented in the \texttt{Julia} programming language 
\cite{bezanson2012julia,
bezanson2014julia}. \texttt{Julia} is a high level language in which
multi-dimensional tables are first class types \cite{bezanson2014array}. For 
purposes of the algorithms, two modules were created. In the first one,
\texttt{SymmetricTensors.jl}, \cite{st} the block structure of super-symmetric
tensors was implemented. In the second module, \texttt{Cumulants.jl}, \cite{cum}
we used the block structure to compute and store moment and cumulant tensors.
The implementation of cumulants calculation uses multiprocessing
primitives built into the \texttt{Julia} programming language:
\emph{remote references} and \emph{remote calls}. A remote reference is an
object that allows any process to reference an object stored in a
specific process. A remote call allows a process to request a function call
on certain arguments on another process.

\section{Performance analysis}\label{sec::performance} This section is dedicated 
to the performance analysis of the core elements of our algorithms. These are
Eq.~\eqref{d::moment} which calculates the moment tensor and
Eq.~\eqref{eq::cum_r} which calculates the cumulant tensor. First, we discuss 
theoretical analysis and then focus on the performance of our 
implementation. In the final subsection, we show how our implementation 
compares to the current state of the art.

\subsection{Theoretical analysis}{\label{ssec::analysis}} We start by discussing the
performance of the moment tensor.  With reference to Section~\ref{sec::moment}, let 
us recall that storage of the moment tensor requires storage of the $b^d
\binom{\bar{n}+d-1}{\bar{n}}$ floating-point numbers. We can approximate $b^d
\binom{\bar{n}+d-1}{\bar{n}} \approx \frac{n^d}{d !}$ for $d \ll
n$~\cite{schatz2014exploiting}. Since we usually calculate cumulants of the 
order of $\leq10$ and we deal with high dimensional data, 
we need approximately $\frac{1}{d !}$ of the computer storage space, compared 
with the na\"ive storage scheme.

As for the cumulants, one should primarily note that the number of elements of 
the inner sum
in Eq.~\eqref{eq::cum_r} in the second line equals the number of set partitions
of $\mathbf{k} = (1, \ldots, d)$ into exactly $\sigma$ parts, such that no part
is of size one, and can be represented as:
\begin{equation}
	\# \{[P_{\sigma}^{(2)}(\mathbf{k})]\} = S'(d, \sigma).
\end{equation}
We call it a modification of the Stirling number of the second kind $S(d, 
\sigma)$,
and compute it as follows
\begin{equation}{\label{eq::n2}}
	\begin{split}
S'(d,1) & = 1 \\
		S'(d, \sigma) & = \frac{d!}{\sigma!} \underbrace{\sum_{d_1 = 2}^{d - 
		2\sigma +
		2} \cdots \sum_{d_r = 2}^{d - 2\sigma + 2r - \sum_{i=1}^{r-1} d_i}
		\cdots}_{\sigma-1} \frac{1}{d_1 ! \cdots d_{r}! \cdots (d -
		\sum_{i=1}^{(\sigma - 1)} d_i)!},
	\end{split}
\end{equation}
where we count the number of ways to divide $d$ elements into subsets of length
$d_1 \geq 2, \ldots, d_r \geq 2, \ldots, d_{\sigma} \geq 2$ such that
$\sum_{r=1}^{\sigma} d_r = d$, and so $d_{\sigma} = d - \sum_{r=1}^{\sigma-1}
d_k$. Factor $\sigma!$ in the denominator counts the number of subset
permutations. Some examples of $S'(d,\sigma)$ are $S'(4,2) = 3$, $S'(5,2) = 10$,
$S'(6,2) = 25$ and $S'(6,3) = 15$.

The following sum
\begin{equation}
\sum_{\sigma = 1}^{\lfloor\frac{d}{2}\rfloor} S'(d, 
\sigma)= 1 + \sum_{\sigma = 2}^{\lfloor\frac{d}{2}\rfloor} S'(d, 
\sigma) = F(d),
\end{equation}
is the number of all partitions of a set of size $d$ into subsets, such that 
there is no subset of size one, and therefore
\begin{equation}{\label{eq::req}}
\begin{split}
F(1)& = 0, \\
F(d+1)& = B(d) - F(d).
\end{split}
\end{equation}
Here $B(d)$ is a Bell number \cite{comtet1974advanced}, the number of all 
partitions of a set of size
$d$ including subsets of size one and $B(d) - F(d)$ is the number of
partitions of a set of size $d$ into subsets such that at least one subset is
of size one. Relation Eq.~\eqref{eq::req} is derived from the fact that there 
is a bijective relation between partitions of $d$ element set into subsets 
such that at least one subset is of size one and partitions of $d+1$ element 
set 
into subsets such that there is no subset of size one.

To compute each element of the inner sum in Eq.~\eqref{eq::cum_r}, 
we need $\sigma - 1$ multiplications, and consequently to compute each element 
of the outer sum in Eq.~\eqref{eq::cum_r}, we need $(\sigma - 1)S'(d, \sigma)$ 
multiplications. Finally, the number of multiplications required to compute the 
second term of the RHS of Eq.~\eqref{eq::cum_r} is
\begin{multline}{\label{eq::n3x}}
N(d) = 
\sum_{\sigma = 2}^{\lfloor\frac{d}{2}\rfloor} (\sigma-1)S'(d, \sigma) \leq \\
\leq \left(\lfloor\frac{d}{2}\rfloor - 1\right)\sum_{\sigma = 
2}^{\lfloor\frac{d}{2}\rfloor} 
S'(d, \sigma)
 = \left(\lfloor\frac{d}{2}\rfloor -1 \right) (F(d)-1) =  U(d),
\end{multline}
for each tensor element. Let us note that $N(4) = 3$, $N(5) = 10$, $N(6) = 55$. 
The plot of $N(d)$ and the upper bound $U(d)$ are shown in Fig.~\ref{fig::f}. 
From 
Fig.~\ref{fig::f} the proposed upper bound produces a very good approximation of the 
number of multiplications.
\begin{figure}[t]
\centering
\includegraphics{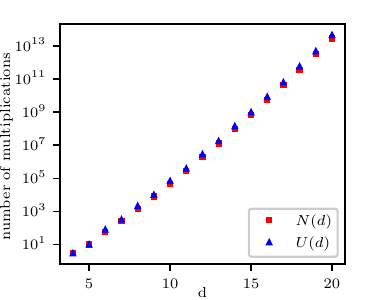}
\label{fig::f}
\caption{Plots of exact number of multiplications $N(d)$, and its upper bound 
based 
on Stirling numbers. Note 
the logarithmic scale on the $y$ axis.}
\end{figure}

From Eq.~\eqref{d::moment}, Eq.~\eqref{eq::cum_r} and Eq.~\eqref{eq::n3x} we can
conclude that to compute the $d$\textsuperscript{th} cumulant's element we need
$(d-1)t$ multiplications for the central moment and $N(d)$ multiplications for
the sums in Eq.~\eqref{eq::cum_r}. However, in order to calculate the cumulant
in an accurate manner, we need large data sets, \ie~for $d=4$ we use $t \sim
10^5$ and for $d > 4$ the data size must be even larger. Bearing in mind that
computation of cumulants of the order of $d > 10$ is inapplicable, the foregoing
gives $(d-1)t\gg N(d)$. Henceforth dominant computational power is required to
calculate the moment tensor, so there appears the need for approximately
$(d-1)t$ multiplications to compute each $d$\textsuperscript{th} cumulant's
element. To compute the whole $d$\textsuperscript{th} cumulant tensor we need
approximately $\frac{(d-1)t n^d}{d!}$ multiplications, while the factor $d!$ is
a result of taking advantage of super-symmetry. The added cost due to blocking
is negligible, see~\cite{schatz2014exploiting}.

It is now possible to compare the complexity of our algorithm with that of the 
na\"ive algorithm for chosen cumulants. For the $4$\textsuperscript{th} 
cumulant, 
the na\"ive algorithm 
would use Eq.~(\ref{eq::c4}) directly and would not take advantage of the 
super-symmetry of tensors. Therefore, it requires $9 t$ multiplications to 
compute a single cumulant tensor element and $9 t n^4$ multiplications to 
compute the whole 
cumulant tensor. Our algorithm, in this case, decreases the computational 
complexity by the factor of $3 \cdot 4! = 72$.

Analogically, the na\"ive formula for the $5$\textsuperscript{th} cumulant
\begin{equation}
	c_{i_1, \ldots, i_5}(\mathbf{X}) =  
	E\bigg(\tilde{X_{i_1}} \ldots \tilde{X_{i_5}}\bigg) 
	-\underbrace{E\bigg(\tilde{X_{i_1}}\tilde{X_{i_2}}\bigg)E\bigg(\tilde{X_{i_3}}
		\tilde{X_{i_4}}\tilde{X_{i_5}}\bigg)
		-\cdots}_{\times 10}
\end{equation}
requires approximately $34 t n^5$ multiplications to compute the whole cumulant
tensor. Our algorithm, in this case, decreases the computational complexity by
the factor of $\frac{34}{4} 5! = 900$. For higher $d$, the difference is even
greater due to the $d!$ factor caused by the application of the block structure
and the fact that the number of terms in na\"ive formulas grows with $d$ much
faster than $F(d)$ from Eq.~\eqref{eq::req}.

\subsection{Implementation performance}{\label{ssec::iperf}}
In this section, we analyse the performance analysis of our implementation. All 
computations were performed in the Prometheus computing cluster. This cluster 
provides shared user access with multiple user tasks running on each node. Each 
node is an HP XL730f Gen9 computing system with dual Intel Xeon E5-2680v3 
processors providing 12 physical cores and 24 computing cores with 
hyper-threading. The node has 128 GB of memory.

\subsubsection{The optimal size of blocks}{\label{ssec::bls}} 
The number of coefficients required to store a super-symmetric tensor of order 
$d$
and $n$ dimensions is equal to $\binom{d+n-1}{n}$. The storage of tensor
disregarding the super-symmetry requires $n^d$ coefficients. The block structure
introduced in~\cite{schatz2014exploiting} uses more than minimal amount of
memory but allows for easier further processing of super-symmetric tensors.

If we store the super-symmetric tensor in the block structure, the block size
parameter $b$ appears. In our implementation in order to store a super-symmetric
tensor in the block structure we need, assuming $n|b$, an array of
$(\frac{n}{b})^d$ pointers to blocks and an array of the same size of flags that
contain the information if a pointer points to a valid block. Recall that
diagonal blocks contain redundant information. Therefore on the one hand,
the smaller the value of $b$, the less redundant elements on diagonals of the
block structure. On the other hand, the larger the value of $b$, the smaller the
number of blocks, the smaller the blocks' operation overhead, and the fewer the
number of pointers pointing to empty blocks. For detailed discussion of memory usage
see~\cite{schatz2014exploiting}. The analysis of the influence of the parameter 
$b$ on the computational time of cumulants for some parameters are
presented in Fig.~\ref{fig::blocks}. We obtain the shortest computation time for
$b = 2$ in almost all test cases, and this value will be set as default and used
in all efficiency tests. Note that for $b=1$ we loose all the memory savings.

\begin{figure}[!h]
	\centering
	\subfigure[$4$\textsuperscript{th} cumulant tensor 
	\label{f::blockf4}]{\includegraphics{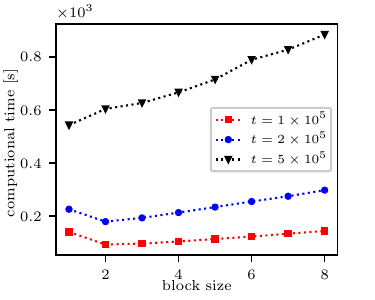}}
	\subfigure[$5$\textsuperscript{th} cumulant tensor
\label{f::blockf5}]{\includegraphics{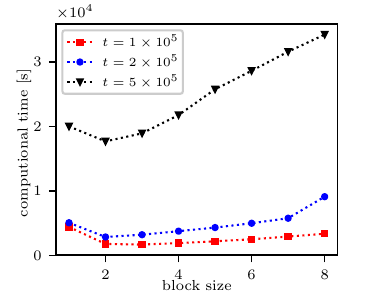}}
	\caption{Computation time for cumulant tensors computed using the block 
	structure and the proposed algorithm, for different block sizes $b$, at $n 
	= 60$.}
	\label{fig::blocks}
\end{figure}

\subsubsection{Comparison with na\"ive algorithms} The computational speedup of
cumulant calculation for the illustrative data is presented in Fig. 
\ref{fig::m4}.
The computational speedup is even higher than the theoretical value of $72$,
which is probably due to large operational memory requirements and some 
computational overhead while splitting data into terms of Eq.~\eqref{eq:c4} 
used by the na\"ive approach.

\begin{figure}[t]
	\centering
	\includegraphics{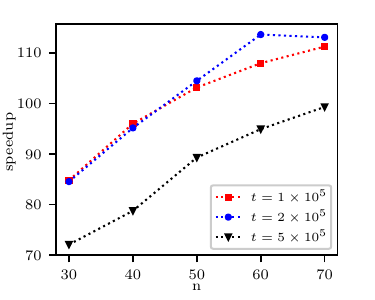}
	\caption{Computation time speed-up of $4$\textsuperscript{th} cumulant 
	tensor computed using the block structure and the proposed algorithm vs the 
	na\"ive algorithm.}
	\label{fig::m4}
\end{figure}

As for the moment calculation, let us recall from Section~\ref{ssec::storage} that we 
expect a speedup on the level of $d!$. As can be shown in 
Fig~\ref{fig::mspeedup}, this is the case for a high number of marginal 
variables $n$, as we approach speedup equal to $24$ for the fourth moment.
This is a case, since there is some redundancy in computation of diagonal 
blocks which decreases as $n$ rises, given $b$.
\begin{figure}[!h]
	\centering
	\includegraphics{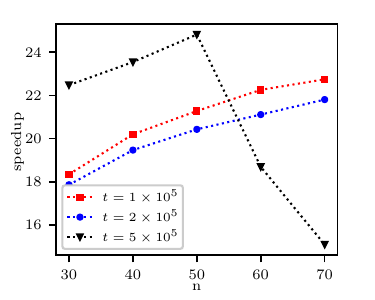}
	\caption{Computation time speedup for the $4$\textsuperscript{th} moment 
		tensor computed using the block structure vs. the na\"ive algorithm, 
		note a speedup up to $24$ times.}
	\label{fig::mspeedup}
\end{figure}

\subsubsection{Multiprocessing performance}

In this section we analyse the multiprocessing performance of moment tensor 
calculations, since according to Subsection~\ref{ssec::analysis}, the moment tensor 
calculation takes the majority of cumulants calculation time. 
Fig.~\ref{fig::nprocs} 
shows the speedup of multiprocess moment 
calculation compared to single process calculation. As can be shown in the 
figure, at first, we obtain linear scaling of the speedup with the number of 
processes. Next, we reach the saturation point. This is expected, as there are 
some parts of this calculation that cannot be done in parallel. Adding more 
processes leads to a drop in the speedup. This is due to the fact that adding 
more processes results in more overall overhead, yet we do not benefit from 
splitting the data further.
\begin{figure}[!h]
\centering

\subfigure[$4$\textsuperscript{th} moment tensor, $t=2\times 
10^5$]{\includegraphics{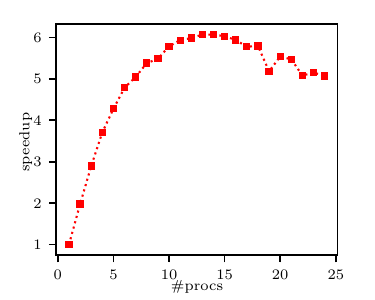}}
\subfigure[$5$\textsuperscript{th} moment tensor, $t=2\times 
10^5$]{\includegraphics{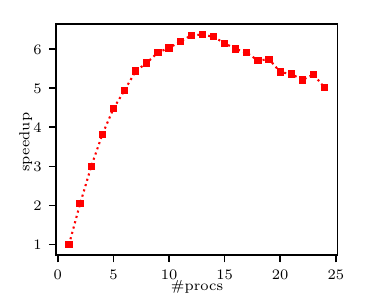}}

\caption{Computation time speedup for the $4$\textsuperscript{th} (left) and the
5\textsuperscript{th} (right) moment tensor due to multiprocessing for $n =
60$.} \label{fig::nprocs}
\end{figure}

\subsection{Comparison with the state of the art}{\label{ssec::soa}} The state of
the art in terms of the cumulant calculation simplification is referred to as 
umbral calculus \cite{rota1994classical}, which is a formal system consisting 
of certain
operations on objects called umbrae, mimicking addition and multiplication of
independent real-valued random variables. Using umbrae notation one can
determine symbolic formulas to calculate elements of cumulant tensors.
See~\cite{di2008unifying} where cumulants, also called $k$-statistics, were
derived using purely combinatorial operations. However, symbolic computations 
are less universal and sometimes problematic, while translating them into 
algorithms and code is not entirely straightforward.

We present a more general approach by implementing an algorithm that takes
multivariate data in the form of a matrix and computes its cumulant tensors. The
current state of the art is a package written in the \texttt{R} programming
language~\cite{de2012multivariate}. This algorithm uses the recursion relation 
to compute
the $d$\textsuperscript{th} cumulant from moments of the order of $1, \ldots, d$
\cite{mccullagh2009cumulants, mccullagh1987tensor, balakrishnan1998note}, see 
Eq.~\eqref{eq::cumfmom}.
\begin{equation}\label{eq::cumfmom}
\CC_{\mathbf{i}}(\mathbf{X}) = \sum_{\sigma = 1}^{d}  
\sum_{\zeta\in 
\{[P_{\sigma}(1 : d)]\}} \left(|\zeta|-1\right)!(-1)^{|\zeta|-1} 
\bigotimes_{\mathbf{k}_r 
\in \zeta} \MM_{(\mathbf{k}_r)}(\mathbf{X}),
\end{equation}
where $|\zeta|$ is the number of parts in the given partition $\zeta$. The
algorithm computes each element of cumulant tensors, without taking advantage of
their super-symmetry. For comparison, our formula, \ie~Eq.~\eqref{eq::cum_r} is
simpler, as it lacks factor $\left(|\zeta|-1\right)!(-1)^{|\zeta|-1}$ and the
inner sum has less elements, since we have introduced $P^{(2)}_{\sigma}$ instead
of $P_{\sigma}$. Further application of Eq.~\eqref{eq::cum_r} enables us to
compute the $d$\textsuperscript{th} cumulant tensor without determining the
$(d-1)$\textsuperscript{th} moment tensor. This fact can be advantageous in
high-resolution direction-finding methods for multi-source signals (the q-MUSIC
algorithm) \cite{chevalier2006high} where one needs a cumulant of the order of
$6$ but not a cumulant and a moment of the order of $5$. Furthermore, the major
benefit of our algorithm is the utilisation of the super-symmetry of cumulant
tensors. By introducing blocks, the computational complexity can be reduced by a
factor of $d!$ in the same manner as the storage requirement is reduced.

In \cite{de2012multivariate} two algorithms were implemented:
one---\texttt{four\_cumulants\_direct}---that uses a direct formula for
cumulants of orders $1$---$4$ which we call the specialized algorithm, and other
one---\texttt{cumulants\_upto\_p}---that can compute cumulants of arbitrary order
using Eq.~\eqref{eq::cumfmom}, which we call the general algorithm. Both
of these algorithms were implemented in the \texttt{R} programming language. The
specialized algorithm outperforms the general one in terms of speed. For
comparison's sake, we re-implemented the general algorithm from
\cite{de2012multivariate} in \texttt{Julia} maintaining high similarity between
both implementations.

To perform the efficiency comparison, we compare the computational time of our
algorithm with the aforementioned algorithms. The obtained results are
summarised in~Fig.~\ref{fig::comparison} which contains:
\begin{itemize}
\item The comparison of our algorithm and the general algorithm~\cite{de2012multivariate} 
re-im\-ple\-men\-ted in \texttt{Julia}---Fig.~\ref{fig::rvsjulia}. 
\item The comparison between our algorithm and the general algorithm implemented
in \cite{de2012multivariate}---Fig.~\ref{fig::rspecial}. Our algorithm is
faster by two orders of magnitude owing to the fact that there exists $d!$
acceleration factor. It results from the utilisation of super-symmetry through
application the block structure. It turns out that our implementation achieves
in practice even higher acceleration.
\item The comparison of our algorithm vs. the specialised algorithm implemented 
in~\cite{de2012multivariate}---Fig.~\ref{fig::rnaive}. 
\end{itemize}

\begin{figure}[t]
\centering
\subfigure[\texttt{Julia} implementation of general algorithm from~\cite{de2012multivariate}.]{\includegraphics{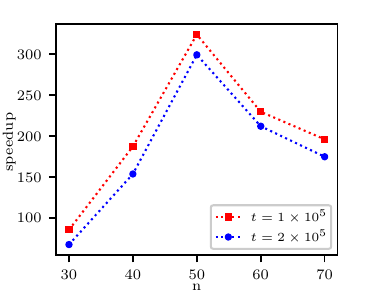}\label{fig::rvsjulia}}
\subfigure[\texttt{R} implementation of general algorithm from~\cite{de2012multivariate}.]{\includegraphics{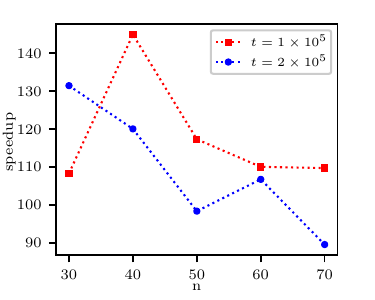}\label{fig::rspecial}}
\subfigure[Specialized algorithm from~\cite{de2012multivariate} implemented in \texttt{R}.]{\includegraphics{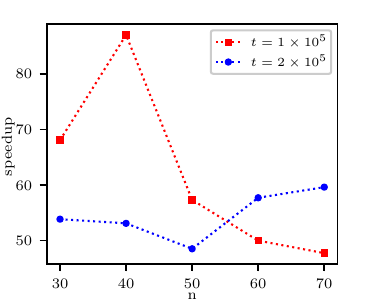}\label{fig::rnaive}}
\caption{Computation time speedup of $4$\textsuperscript{th} cumulant tensor calculation using algorithm employing the block structure vs. algorithms from~\cite{de2012multivariate}.}
\label{fig::comparison}
\end{figure}

\section{Conclusions} This paper provides a discussion on both the method and 
the algorithm for calculation of arbitrary order moment and cumulant tensors 
given multidimensional data. To this end, we introduce the recurrence relation 
between
the $d$\textsuperscript{th} cumulant tensor and the $d$\textsuperscript{th} 
central moment
tensor as well as cumulant tensors of the order of $2, \ldots, d-2$. For 
purposes of  efficient computation and storage of super-symmetric tensors, we 
use blocks to store and
calculate only the pyramidal part of cumulant and moment tensors. Our algorithm
is significantly faster than the existing algorithms. The theoretical speedup is
given by the factor of $d!$, which makes the algorithm applicable in the 
analysis of large data sets. 
Another important aspect is that large data
sets are required to approximate accurately high order statistics on account of 
their large approximation error. If the estimation error challenge is 
successfully tackled, high order multidimensional statistics such as high order 
moments or cumulants will be an
important tool to analyse non-normally distributed data, where the mean vector
and the covariance matrix contain little information about the data. There are
many applications of such statistics, particularly involving signal analysis,
financial data analysis, hyper-spectral data analysis or particle physics.

\appendix
\section{The estimation error of high order statistics}\label{app::estimation} Let
$M_d$ be an estimator of the $d$\textsuperscript{th} moment of one-dimensional
centered random variable $V$, and let us have available $t$ realisations of $V$.
As we consider large $t$, the bias of such an estimator can be neglected as
being much smaller than a standard error. Hence, we can use the following
estimator
\begin{equation}
	M_d = \frac{1}{t} \sum_{l = 1}^t (V_l)^d,
\end{equation}
where we just sum $t$ independent random variables $V_l$ raised to the power of
$d$. The variance of $M_d$ can be represented as:
\begin{equation}
	\text{var}(M_d) = \frac{1}{t^2}\text{var}\left(\sum_{l =1}^t 
	\left(V_l\right)^d\right).
\end{equation}
Since $V_1, \ldots, V_t$ are independent and equal in distribution to $V$,
\begin{equation}
	\text{var}\left(\sum_{l =1}^t (V_l)^d\right) = \sum_{l =1}^t 
	\text{var}\left(V^d\right) = t \  \text{var}\left( 
	V^d\right),
\end{equation} 
hence
\begin{equation} 
	\text{var}(M_d) = \frac{1}{t}\left(M_{2d} - 
	(M_d)^2\right) < \frac{M_{2d}}{t}, \ \ \ \ \text{std}(M_d) 
	<\sqrt{\frac{M_{2d}}{t}},
\end{equation}
and obviously this limit is relevant if $M_{2d}$ exists. In the multivariate
case $V_1, \ldots, V_t$ are only independent in groups. The number of groups can
be estimated using the number of marginal variables $n$, but still $n \ll t$.
Consequently, a similar limitation can be expected, but replacing $M_{2d}$ with
a product of moments of lower orders.

\section{The recurrence formula for cumulant calculations}\label{app::recurrence} We
recall the cumulant generating function
\begin{equation}\label{eq:cum_gen}
K(\tau) = 
\log\left(\frac{\sum_{l=1}^t\exp\left(	\left[x_{l,1}, \ldots,  
x_{l,n} \right] \cdot\tau \right)}{t}\right),
\end{equation} 
which is related to the moment generation (characteristic) function
$\tilde{\phi}(\tau)$, $K(\tau) = \log(\tilde{\phi}(\tau)).$ For simplicity, we
use the following notation: $\partial_i = \frac{\partial}{\partial \tau_i}$,
$\partial_{\mathbf{i}} = \partial_{i_1, \ldots, i_d} =
\frac{\partial^d}{\partial \tau_{i_1}, \ldots, \partial \tau_{i_d}}$, and drop
$\mathbf{X}$ in notation $c(\mathbf{X}) \rightarrow c$. The elements of the
moment and cumulant tensor at multi-index $\mathbf{i}$ are
\begin{equation}\label{eq::moments}
	m_{\mathbf{i}}(\mathbf{X}) = \partial_{\mathbf{i}} \tilde{\phi}(\tau)
	\big|_{\tau = 0} \ \ \text{and} \ \ c_{\mathbf{i}}(\mathbf{X}) =
	\partial_{\mathbf{i}}  K(\tau) \big|_{\tau = 0}.
\end{equation}
We have the following theorem
\begin{proposition}
For each $\mathbf{i}$ the following holds:
	\begin{equation}\label{eq::proofT}
		\frac{\partial_{\mathbf{i}}\tilde{\phi}(\tau)}{\tilde{\phi}(\tau)} = 
		\sum_{\sigma = 1}^{|\mathbf{i}|} \sum_{\zeta\in \{[P_{\sigma}(1 : 
		d)]\}} \prod_{\mathbf{k}_r \in \zeta} 
		c_{\mathbf{i}_{\mathbf{k}_r}}(\tau).  
	\end{equation}
\end{proposition}
\begin{proof}
For $|\mathbf{i}| = 1$ the results follow from direct inspection. Next, for
$|\mathbf{i}| =2$ we get:
\begin{equation}\label{eq::proofm2}
 c_{i_1, i_2}(\tau) =
 \partial_{i_{1}}\left(\frac{\partial_{i_{2}}\tilde{\phi}(\tau)}{\tilde{\phi}(\tau)}\right)
  = \frac{ \partial_{i_{2} i_{1}} \tilde{\phi}(\tau)}{\tilde{\phi}(\tau)} - 
 c_{i_{2}}(\tau)c_{i_{1}}(\tau).
\end{equation} Now assume that 
Eq.~(\ref{eq::proofT}) holds for $|\mathbf{i}| = d$. Differentiating its LHS, we
have
\begin{equation} \partial_{i_{d+1}} 
\frac{\partial_{\mathbf{i}}\tilde{\phi}(\tau)}{\tilde{\phi}(\tau)} = 
\frac{\tilde{\phi}(\tau) \partial_{\mathbf{i}} 
\partial_{i_{d+1}}\tilde{\phi}(\tau)- 
\partial_{i_{d+1}}\tilde{\phi}(\tau)\partial_{\mathbf{i}}\tilde{\phi}(\tau)}{\tilde{\phi}^2(\tau)},
\end{equation}
further using Eq.~(\ref{eq::proofT}) we obtain
$\partial_{i_{d+1}}\tilde{\phi}(\tau) = \tilde{\phi}(\tau) c_{i_{d+1}}(\tau)$, 
therefore 
\begin{equation}{\label{eq::dlhs}}
 \partial_{i_{d+1}} 
\frac{\partial_{\mathbf{i}}\tilde{\phi}(\tau)}{\tilde{\phi}(\tau)} = \frac{ 
\partial_{\mathbf{i}'}\tilde{\phi}(\tau)}{\tilde{\phi}(\tau)} - 
c_{i_{d+1}}(\tau) \sum_{\sigma = 1}^{|\mathbf{i}|} \sum_{\zeta\in 
[P_{\sigma}(1 : d)]} \prod_{\mathbf{k}_r \in \zeta} 
c_{\mathbf{i}_{\mathbf{k}_r}}(\tau),
\end{equation}
where $\mathbf{i}' = (\mathbf{i}, i_{d+1})$. After differentiating 
Eq.~\eqref{eq::moments}, we have 
\begin{equation} 
\partial_{i_{d+1}}c_{\mathbf{i}}(\tau)  = c_{(\mathbf{i}, i_{d+1})}(\tau), 
\end{equation}
and analogously 
\begin{equation} 
\partial_{i_{d+1}}\prod_{\mathbf{k}_r \in \zeta} 
c_{\mathbf{i}_{\mathbf{k}_r}}(\tau) = \sum_{a=1}^{\sigma} 
c_{(\mathbf{i}_{\mathbf{k}_a}, i_{d+1})}(\tau)
\prod_{\mathbf{k}_r \in \zeta, r \neq a} c_{\mathbf{i}_{\mathbf{k}_r}} 
(\tau).
\end{equation}
Differentiating the RHS of Eq.~(\ref{eq::proofT}),
\begin{equation}{\label{eq::drhs}} 
\begin{split} 
\partial_{i_{d+1}} 
\sum_{\sigma = 1}^{|\mathbf{i}|} \sum_{\zeta\in \{[P_{\sigma}(1 : d)]\}} 
\prod_{\mathbf{k}_r \in \zeta} c_{\mathbf{i}_{\mathbf{k}_r}}(\tau) = 
\sum_{\sigma = 1}^{|\mathbf{i}|} \sum_{\zeta\in 
\{[P_{\sigma}(1 : d)]\}}\sum_{a=1}^{\sigma}c_{(\mathbf{i}_{\mathbf{k}_a}, 
i_{d+1})}(\tau) \prod_{\substack{\mathbf{k}_r 
\in \zeta\\ r \neq a}} c_{\mathbf{i}_{\mathbf{k}_r}} (\tau),
\end{split}
\end{equation}
comparing Eq.~\eqref{eq::drhs} with Eq.~\eqref{eq::dlhs}, we have 
\begin{equation} \begin{split} \frac{ 
\partial_{(\mathbf{i}, i_{d+1})}\tilde{\phi}(\tau)}{\tilde{\phi}(\tau)} &= 
c_{i_{d+1}}(\tau) \sum_{\sigma = 1}^{|\mathbf{i}|} \sum_{\zeta\in 
\{[P_{\sigma}(1 : d)]\}} \prod_{\mathbf{k}_r \in \zeta} 
c_{\mathbf{i}_{\mathbf{k}_r}}(\tau) \\ &+ \sum_{\sigma = 
1}^{|\mathbf{i}|} \sum_{\zeta\in \{[P_{\sigma}(1 : d)]\}} 
\sum_{a=1}^{\sigma}c_{(\mathbf{i}_{\mathbf{k}_a}, 
	i_{d+1})}(\tau)\prod_{\mathbf{k}_r \in \zeta, r \neq a} 
c_{\mathbf{i}_{\mathbf{k}_r}} (\tau). \end{split} 
\end{equation} 
Finally, we obtain
\begin{equation} \frac{ 
\partial_{\mathbf{i}'}\tilde{\phi}(\tau)}{\tilde{\phi}(\tau)} = \sum_{\sigma = 
1}^{|\mathbf{i}'|} \sum_{\zeta\in \{[P_{\sigma}(1 : (d+1))]\}} 
\prod_{\mathbf{k}_r \in \zeta} c_{\mathbf{i}_{\mathbf{k}_r}}(\tau).
\end{equation}
\end{proof}
If we observe that $\tilde{\phi}(\tau)\big|_{\tau = 0} = 1$ and $m_{\mathbf{i}}
= \partial_{\mathbf{i}} \tilde{\phi}(\tau) \big|_{\tau = 0}$ and
$c_{\mathbf{i}}(\tau) \big|_{\tau = 0} = c_{\mathbf{i}}$, then
Eq.~(\ref{eq::proofT}) at $\tau = 0$ will give Eq.~(\ref{eq::cum}).

\section*{Acknowledgements}
The authors would like to thank Adam Glos for revising the manuscript and 
Zbigniew Puchała for the discussion about error estimation and set partitions.
This research was supported in part by PL-Grid Infrastructure
\bibliographystyle{siamplain}
\bibliography{tensor_network}
\end{document}